\documentclass[letterpaper, 10 pt, conference]{ieeeconf}  

\IEEEoverridecommandlockouts                              
\usepackage{additionalSetting}
\usepackage{anyfontsize}



\title{\LARGE \bf
On moment relaxations for linear state feedback controller synthesis with non-convex quadratic costs and constraints
}

\author{Dennis Gramlich, Sheng Gao, Hao Zhang, Carsten W. Scherer and Christian Ebenbauer
\thanks{Carsten W. Scherer is funded by Deutsche Forschungsgemeinschaft (DFG, German Research Foundation) under Germany's Excellence Strategy - EXC 2075 – 390740016. He acknowledges the support by the Stuttgart Center for Simulation Science (SimTech).}
\thanks{Dennis Gramlich and Christian Ebenbauer are with the Chair of Intelligent Control Systems,
        RWTH Aachen University,
        52074 Aachen, Germany
        {\tt\small \{dennis.gramlich,christian.ebenbauer\} @ic.rwth-aachen.de}}%
\thanks{Sheng Gao, and Hao Zhang are with the College of Electronic and Information Engineering, the Department of Control Science and Engineering, Tongji University, Shanghai, 200092, P. R. China. {\tt\small 2110134@tongji.edu.cn,zhang\_hao@tongji.edu.cn.}}
\thanks{Carsten W. Scherer is with the Chair of Mathematical Systems Theory, University of Stuttgart, 70569 Stuttgart, Germany {\tt\small carsten.scherer@imng.uni-stuttgart.de}}%
}%

\makeatletter
\newcommand*{\compress}{\@minipagetrue}
\makeatother

\newenvironment{ilist}%
{
	\setdefaultleftmargin{1em}{1em}{}{}{}{} 
	\compress 
	\begin{itemize}%
		\setlength{\itemsep}{0pt}%
		\setlength{\topsep}{1pt} 
		\setlength{\partopsep}{0pt}
		\setlength{\parsep}{0pt}
		\setlength{\parskip}{0pt}}%
	{\end{itemize}}

\usepackage{etoolbox}
\makeatletter
\patchcmd{\ps@pprintTitle}
{Preprint submitted}
{To be submitted}
{}{}
\makeatother

\begin{document}

\maketitle
\thispagestyle{empty}
\pagestyle{empty}

\begin{abstract}


We present a simple and effective way to account for non-convex costs and constraints~in~state feedback synthesis, and an interpretation for the variables in which state feedback synthesis is typically convex. We achieve this by deriving the controller design using moment matrices of state and input. It turns out that this approach allows the consideration of non-convex constraints by relaxing them as expectation constraints, and that the variables in which state feedback synthesis is typically convexified can be identified with blocks of these moment matrices.



\end{abstract}

\section{INTRODUCTION}
\label{sec:1}

To derive convex controller synthesis conditions subject to potentially non-convex quadratic constraints we take a similar approach as \cite{gattami2009generalized}. This approach uses moment matrices and relaxes (non-convex) quadratic constraints as expectation constraints (power constraints). Due to the similarities to \cite{gattami2009generalized}, we highlight the novelties of this paper/differences to \cite{gattami2009generalized} right at the beginning.
\begin{ilist}
    \item We show that, as in non-convex quadratic programming \cite{park2017general}, the relaxation of non-convex~quadratic constraints as expectation constraints yields a deterministic solution, for some problems, and a stochastic solution, for other problems. Deterministic means that the optimal policy of the relaxed problem is deterministic and coincides with an optimal policy of the original problem. Stochastic means that the optimal policy is stochastic and not optimal for the original problem.
    \item We show how to realize the stochastic optimal policies.
    \item We identify blocks of our moment matrices with the variables in which state feedback synthesis is usually convexified using the dualization lemma \cite{scherer2000robust}.
    \item We study moment matrices of affine linear systems.
    \item We explain the benefits of linear controller synthesis with non-convex constraints in compelling examples.
\end{ilist}

Convexification of state feedback synthesis is considered to be a relatively well understood problem in control theory. Lossless convexifications, i.e., transformations that convert a generally non-convex optimization problem into a convex optimization problem, exist for the quadratic stabilization of linear systems \cite{4793262}, $H_2$ and $H_\infty$ state feedback synthesis \cite{gahinet1994linear}, robust state feedback synthesis for systems in linear fractional representation form \cite{scherer2000linear}, and data-based controller synthesis \cite{8960476}.
The equivalent convex problems are usually obtained by working with the \emph{dual system} \cite{KALMAN1960491}, or by invoking a \emph{dualization lemma} \cite{scherer2000robust}. Throughout, a variable transformation from the parameters of a quadratic storage function to new parameters is used, since the controller synthesis is non-convex in the original parameters for the listed problems.

All the cited studies consider the new parameters merely as a useful tool for finding the actual parameters of interest. In the present work, however, we give~the transformed parameters an interpretation by deriving the convexification based on moment matrices of state and input. In the moment matrices, state feedback synthesis is convex and 
the blocks of these moment matrices can be identified with transformed parameters of other works.

Studying the moment matrices of state and input of a dynamical system is certainly an established approach in the research literature. The latter comes from the fact that occupation measures \cite{pitman1977occupation} and Lyapunov measures \cite{rantzer2001dual,vaidya2008lyapunov} offer a dual perspective to value functions and Lyapunov functions with advantages for the convexification of state feedback synthesis. It is therefore proposed, e.g., in \cite{lasserre2008nonlinear,holtorf2022stochastic}, to optimize over the moment matrices of occupation measures of state and input. The restriction to linear systems and moments of order two, which we consider in this paper, is a special case of what is considered in these studies. On the other hand, linear systems are generally recognized as an important case and allow for stronger results than polynomial systems. Controller synthesis for linear systems via deterministically interpreted second-order moment matrices is discussed in detail in \cite{bamieh2024linear}. As shown in \cite{gattami2015simple}, it enables the convex computation of the $H_\infty$-norm of linear systems. In addition, state feedback LQG synthesis via second-order moment matrices is discussed in \cite{kamgarpour2017infinite}. In contrast to \cite{bamieh2024linear,kamgarpour2017infinite}, the present paper and \cite{gattami2009generalized} address control problems with non-convex quadratic costs and constraints. These are convex in the moment matrices and always admit the extraction of a stochastic optimal solution. However,~in the presence of non-convex costs or constraints, the optimal moment matrices may be realizable only by a stochastic policy satisfying constraints only in expectation.

Moment matrices can have various benefits for the convexification structural controller constraints \cite{1656453,causevic2022optimal}, or multi-objective control problems \cite{lee2023multi}. We expand this list with the design affine linear controllers that avoid non-convex regions in the state space (Section \ref{sec:6}) or swing up an inverted pendulum (Section \ref{sec:7}).

\section{Problem statement}
\label{sec:2}

We study a discrete-time linear time varying system
\begin{align}
	x_{\ti +1} = f_\ti + A_\ti x_\ti + B_\ti u_\ti + w_\ti. \label{eq:systemDynamics}
\end{align}
Here, $x_\ti \in \bbR^n$ is the state, $u_\ti \in \bbR^m$ is the control input and $w_\ti \in \bbR^n$ is a disturbance. Accordingly, $A_\ti \in \bbR^{n\times n}$ and $B_\ti \in \bbR^{n\times m}$ are the system matrices at time $\ti$ and $f_\ti \in \bbR^n$ is an additional constant term in the dynamics.

For \eqref{eq:systemDynamics} we study the stochastic control problem
\begin{subequations}
	\label{eq:timeVaryingProblem}
\begin{align}
	\minimize_{(\bbP_{u_\ti})} ~&~ \bbE \sum_{\ti =0}^N \begin{pmatrix}
		1\\
		x_\ti\\
		u_\ti
	\end{pmatrix}^\top R_\ti
	\begin{pmatrix}
		1\\
		x_\ti\\
		u_\ti
	\end{pmatrix} \label{eq:timeVaryingProblem1}\\
	\mathrm{s.t.} ~&~ x_{\ti +1} = f_\ti + A_\ti x_\ti + B_\ti u_\ti + w_\ti ,\\
	~&~ \bbE \begin{pmatrix}
		1\\
		x_\ti\\
		u_\ti
	\end{pmatrix}^\top H_{\ti i}
	\begin{pmatrix}
		1\\
		x_\ti\\
		u_\ti
	\end{pmatrix} \leq 0, \quad i = 1,\ldots,s, \label{eq:timeVaryingProblem3}\\
    ~&~ u_\ti \sim \bbP_{u_\ti},
    ~~~ w_\ti \sim \bbP_{w_\ti},
	~~~ x_0 \sim \bbP_{x_0}.
\end{align}
\end{subequations}
In \eqref{eq:timeVaryingProblem}, $\bbP_{x_0}$ denotes the probability distribution of the initial state, $\bbP_{w_\ti}$ denotes the distribution of the disturbance at time $\ti$ and $\bbP_{u_\ti}$ denotes the distribution of the input at time $\ti$. We assume that the first and second moments of $x_0$ are fixed and known and denote them with
\begin{align}
    \ovl{\Sigma}_0
    =
	\begin{pmatrix}
		\sigma_0^{11} & \sigma_0^{12}\\
		\sigma_0^{21} & \Sigma_0^{22}
	\end{pmatrix}
	=
    \bbE \begin{pmatrix}
		1\\
		x_0
	\end{pmatrix}
	\begin{pmatrix}
		1\\
		x_0
	\end{pmatrix}^\top . \label{eq:initialMoments}
\end{align}
The initial state $x_0$ and the disturbances $w_\ti$ for $\ti = 0,\ldots,N-1$ are assumed to be stochastically independent. In addition, the control input $u_\ti$ is restricted to be conditionally independent from $x_0$ and $w_\ti$ for $\ti = 0,\ldots,N-1$ given the current state $x_\ti$. The latter restriction makes sure that the closed loop is a Markov process \cite{howard2012dynamic} and is satisfied, e.g., if $u_\ti = \pi_\ti (x_\ti ,v_\ti )$, where $\pi_t$ is some measurable function and $v_\ti$ is a random variable independent from $x_0$ and $w_\ti$ for $\ti = 0,\ldots,N-1$. The matrices $R_\ti,H_{\ti i} \in \bbR^{(1+n+m) \times (1+n+m)}$ are assumed to be symmetric, but may be indefinite. In \eqref{eq:timeVaryingProblem} we regard the distributions $\bbP_{u_\ti}$ of the control inputs $u_\ti$ as the control policy and we minimize over this control policy. Problem \eqref{eq:timeVaryingProblem} could be a relaxation of a deterministic control problem with non-convex quadratic constraints, which is np-hard.

\section{Convexification over moment matrices}
\label{sec:3}

To approach \eqref{eq:timeVaryingProblem}, we study the moment matrices
\begin{align}
	\Sigma_\ti =
	\begin{pmatrix}
		\sigma_\ti^{11} & \sigma_\ti^{12} & \sigma_\ti^{13}\\
		\sigma_\ti^{21} & \Sigma_\ti^{22} & \Sigma_\ti^{23}\\
		\sigma_\ti^{31} & \Sigma_\ti^{32} & \Sigma_\ti^{33}
	\end{pmatrix}
	:=
	\bbE \begin{pmatrix}
		1\\
		x_\ti\\
		u_\ti
	\end{pmatrix}
	\begin{pmatrix}
		1\\
		x_\ti\\
		u_\ti
	\end{pmatrix}^\top \label{eq:momentEquation}
\end{align}
of a random trajectory $(x_\ti,u_\ti)$ of \eqref{eq:systemDynamics}, where $x_0 \sim \bbP_{x_0}$, $w_\ti \sim \bbP_{w_\ti}$ and $u_\ti \sim \bbP_{u_\ti}$ for $\ti = 0,\ldots,N$. Obviously, each control policy $(\bbP_{u_\ti})$ gives rise to a sequence of moments $(\Sigma_\ti)$. The advantage of studying the moments $(\Sigma_\ti)$ lies in the fact that the sequence of moments that can result from a control policy $(\bbP_{u_\ti})$ is characterized by simple linear and semi-definite constraints. These constraints involve the matrix valued functions $(\Sigma,\Sigma_+,\Sigma^w)\mapsto \widetilde{F}_\ti(\Sigma,\Sigma_+,\Sigma^w)$ defined by $\widetilde{F}_\ti(\Sigma,\Sigma_+,\Sigma^w)$ equal to
\begin{align*}
	(\bullet)^\top
	\begin{pmatrix}
		-\sigma^{11} & -\sigma^{12} & -\sigma^{13}\\
		-\sigma^{21} & -\Sigma^{22} & -\Sigma^{23}\\
		-\sigma^{31} & -\Sigma^{32} & -\Sigma^{33}\\
		& & & -\Sigma^w\\
		& & & & \sigma_+^{11} & \sigma_+^{12}\\
		& & & & \sigma_+^{21} & \Sigma_+^{22}
	\end{pmatrix}
	\begin{pmatrix}
		1 & f_\ti^\top\\
		0 & A_\ti^\top\\
		0 & B_\ti^\top\\
		0 & I\\
		1 & 0\\
		0 & I
	\end{pmatrix}.
\end{align*}

\begin{theorem} \label{thm:controllerParametrization}
	Let $\Sigma_\ti = \Sigma_\ti^\top \in \bbR^{(1+n+m) \times (1+n+m)}$ be a sequence of matrices
    satisfying the initial condition \eqref{eq:initialMoments}. Then there exists a policy $(\bbP_{u_\ti})$ generating the sequence of moments $(\Sigma_\ti)$ satisfying \eqref{eq:momentEquation} if and only if
	\begin{align}
        \Sigma_\ti &\succeq 0, & \widetilde{F}_\ti(\Sigma_\ti,\Sigma_{\ti +1},\Sigma_\ti^w) &= 0 \label{eq:cond}
	\end{align}
	hold for $\ti = 0,\ldots,N$. Moreover, if \eqref{eq:cond} holds, then there exist controller parameters $K_\ti = \begin{pmatrix}
        k_\ti^1 & K_\ti^2
    \end{pmatrix}$ and $\Sigma_\ti^v$ with
    \begin{align}
        \begin{pmatrix}
            \sigma_\ti^{31} & \Sigma_\ti^{32}
        \end{pmatrix}
        &=
        \begin{pmatrix}
            k_\ti^1 & K_\ti^2
        \end{pmatrix}
        \begin{pmatrix}
			\sigma_\ti^{11} & \sigma_\ti^{12}\\
			\sigma_\ti^{21} & \Sigma_\ti^{22}
		\end{pmatrix} \label{eq:controllerParamDef1}\\
        \Sigma_\ti^v &= \Sigma_\ti^{33} - \begin{pmatrix}
            k_\ti^1 & K_\ti^2
        \end{pmatrix} \begin{pmatrix}
            \sigma_\ti^{31} & \Sigma_\ti^{32}
        \end{pmatrix}^\top \succeq 0 \label{eq:controllerParamDef2}
    \end{align}
    and the moments $(\Sigma_\ti)$ result from the particular policy
	\begin{align}
		u_\ti &= k_\ti^1 + K_\ti^2 x_\ti + v_\ti, \label{eq:controllerReconstruction}
	\end{align}
	where $v_\ti \sim \bbP_{v_\ti}$ are some random variables independent from $x_0, w_\ti$ for $\ti = 0,\ldots,N-1$ with zero mean and variance $\Sigma_\ti^v$.
\end{theorem}
\begin{proof}
	Let a policy $(\bbP_{u_\ti})$ be given and let $(\Sigma_\ti)$ be the sequence of moment matrices. We start by showing that \eqref{eq:cond} must hold. To this end, consider the equation
	\begin{align*}
		&\begin{pmatrix}
			1 & 0 & 0\\
			f_\ti & A_\ti & B_\ti
		\end{pmatrix}
		\begin{pmatrix}
			\sigma_\ti^{11} & \sigma_\ti^{12} & \sigma_\ti^{13}\\
			\sigma_\ti^{21} & \Sigma_\ti^{22} & \Sigma_\ti^{23}\\
			\sigma_\ti^{31} & \Sigma_\ti^{32} & \Sigma_\ti^{33}
		\end{pmatrix}
		(\bullet)^\top
        +
        \begin{pmatrix}
            0 & 0\\
            0 & \Sigma^w
        \end{pmatrix}\\
		&\overset{\eqref{eq:momentEquation}}{=}
			\bbE \begin{pmatrix}
				1 & 0 & 0 & 0\\
				f_\ti & A_\ti & B_\ti & I
			\end{pmatrix} \begin{pmatrix}
			1\\
			x_\ti\\
			u_\ti\\
            w_\ti
		\end{pmatrix}
        \begin{pmatrix}
			1\\
			x_\ti\\
			u_\ti\\
            w_\ti
		\end{pmatrix}^\top
        (\bullet)^\top\\
		&\overset{\eqref{eq:systemDynamics}}{=}
		\bbE \begin{pmatrix}
			1\\
			x_{\ti +1}
		\end{pmatrix}
		\begin{pmatrix}
			1\\
			x_{\ti +1}
		\end{pmatrix}^\top
		\overset{\eqref{eq:momentEquation}}{=}
		\begin{pmatrix}
			\sigma_{\ti +1}^{11} & \sigma_{\ti +1}^{12}\\
			\sigma_{\ti +1}^{21} & \Sigma_{\ti +1}^{22}
		\end{pmatrix},
	\end{align*}
	which follows from \eqref{eq:momentEquation}, \eqref{eq:systemDynamics} and the independence of $w_\ti$ from $(x_\ti,u_\ti)$. This equation implies $\widetilde{F}_\ti(\Sigma_\ti,\Sigma_{\ti +1},\Sigma_\ti^w) = 0$. The constraints $\Sigma_\ti \succeq 0$ must obviously hold. Hence, \eqref{eq:cond} is satisfied for any control policy $(\bbP_{u_\ti})$.
	
	Next, let a sequence of matrices $(\Sigma_\ti)$ satisfying \eqref{eq:cond} be given. We show that these matrices equal the moment matrices under the policy \eqref{eq:controllerReconstruction}. To this end, assume that
	\begin{align}
		\bbE \begin{pmatrix}
			1\\
			x_\ti
		\end{pmatrix}
		\begin{pmatrix}
			1\\
			x_\ti
		\end{pmatrix}^\top 
		=
		\begin{pmatrix}
			\sigma_\ti^{11} & \sigma_\ti^{12}\\
			\sigma_\ti^{21} & \Sigma_\ti^{22}
		\end{pmatrix} \label{eq:inductionHypothesis}
	\end{align}
	and \eqref{eq:controllerReconstruction} hold at time $\ti$. Then, due to $\Sigma_\ti \succeq 0$, there exist controller parameters $K_\ti = \begin{pmatrix}
        k_\ti^1 & K_\ti^2
    \end{pmatrix}$ and $\Sigma_\ti^v = \Sigma_\ti^{33} - k_\ti^1 \sigma_\ti^{13} - K_\ti^2\Sigma_\ti^{23} \succeq 0$ satisfying \eqref{eq:controllerParamDef1} and \eqref{eq:controllerParamDef2} such that a random variable $v_\ti \sim \bbP_{v_\ti}$ with zero mean and variance $\Sigma_\ti^v$ can be defined. Next, due to $\bbE v_\ti = 0$, $\bbE v_\ti x_\ti^\top = 0$ and \eqref{eq:controllerReconstruction}, we can show
   \begin{align*}
		\bbE \begin{pmatrix}
		    1\\
            x_\ti\\
            u_\ti
		\end{pmatrix}u_\ti^\top 
        &=
        \begin{pmatrix}
            1 & 0 & 0\\
            0 & I & 0\\
            k_\ti^1 & K_\ti^2 & I
        \end{pmatrix}
        \bbE \begin{pmatrix}
            1\\
            x_\ti\\
            v_\ti
        \end{pmatrix}
        \begin{pmatrix}
            1\\
            x_\ti\\
            v_\ti
        \end{pmatrix}^\top 
        \begin{pmatrix}
            (k_\ti^1)^\top\\ (K_\ti^2)^\top\\ I
        \end{pmatrix}
        \\
        &=
        \begin{pmatrix}
            1 & 0 & 0\\
            0 & I & 0\\
            k_\ti^1 & K_\ti^2 & I
        \end{pmatrix}
        \begin{pmatrix}
            \sigma_\ti^{11} & \sigma_\ti^{12} & 0\\
            \sigma_\ti^{21} & \Sigma_\ti^{22} & 0\\
            0 & 0 & \Sigma_\ti^v
        \end{pmatrix}
        \begin{pmatrix}
            (k_\ti^1)^\top\\ (K_\ti^2)^\top\\ I
        \end{pmatrix}\\
        &=
        \begin{pmatrix}
            1 & 0 & 0\\
            0 & I & 0\\
            k_\ti^1 & K_\ti^2 & I
        \end{pmatrix}
        \begin{pmatrix}
            \sigma_\ti^{13}\\ \Sigma_\ti^{23}\\ \Sigma_\ti^v
        \end{pmatrix}
        = \begin{pmatrix}
            \sigma_\ti^{13}\\
            \Sigma_\ti^{23}\\
            \Sigma_\ti^{33}
        \end{pmatrix}.
	\end{align*}
	Combining the above with \eqref{eq:inductionHypothesis} shows that \eqref{eq:momentEquation} holds at time $\ti$. Next, \eqref{eq:momentEquation} and $\widetilde{F}_\ti(\Sigma_\ti,\Sigma_{\ti +1},\Sigma_\ti^w) = 0$ together imply \eqref{eq:inductionHypothesis} at time $\ti+1$. In this way, we can prove inductively that $(\Sigma_\ti)$ satisfies \eqref{eq:momentEquation} for all times $\ti$.
\end{proof}
Due to Theorem \ref{thm:controllerParametrization}, we can directly optimize over the matrix sequence $(\Sigma_\ti)$ instead of the policy $(\bbP_{u_\ti})$. To this end, notice that the cost \eqref{eq:timeVaryingProblem1} and constraints \eqref{eq:timeVaryingProblem3} can easily be expressed in terms of $\Sigma_\ti$ as
\begin{align*}
	\sum_{\ti =0}^N \trace \Sigma_\ti R_\ti && \text{and} && 
	\trace \Sigma_\ti H_{\ti i} \leq 0.
\end{align*}
Hence, we can reformulate \eqref{eq:timeVaryingProblem} as the convex program
\begin{align}
	\minimize_{(\Sigma_\ti)} ~&~ \sum_{\ti =0}^N \trace \Sigma_\ti R_\ti \label{eq:timeVaryingSynthesis}\\
	\mathrm{s.t.} ~&~ \widetilde{F}(\Sigma_\ti,\Sigma_{\ti +1},\Sigma_\ti^w) = 0, & \ti = 0,\ldots,N-1, \nonumber\\
	~&~ \trace \Sigma_\ti H_{\ti i} \leq 0, & \ti = 0,\ldots, N, i = 1,\ldots,s,\nonumber\\
	~&~ \Sigma_\ti \succeq 0, & \ti = 0,\ldots,N. \nonumber
\end{align}

\begin{remark}
    We mention that in the case $\Sigma_\ti^v = 0$ for $\ti = 0,\ldots,N$, the optimal policy is deterministic and otherwise it is stochastic. As it is shown in \cite{kamgarpour2017infinite}, if $R_\ti$ and $H_{\ti i}$ are positive semi-definite for all $\ti$ and $i$, then the optimal policy will be deterministic.
\end{remark}

\begin{remark}
    Consider \cite{gattami2009generalized} to see how to combine the state feedback synthesis presented in this section with a Kalman filter to obtain optimal output feedback policies.
\end{remark}

\section{The time-invariant case}
\label{sec:4}

We can also study the time-invariant infinite time-horizon version of \eqref{eq:timeVaryingProblem} defined as
\begin{subequations}
	\label{eq:timeInvProblem}
	\begin{align}
		\minimize_{\bbP_u} ~&~ \bbE
		\lim_{N\to \infty} \frac{1}{N} \sum_{\ti =0}^N
		\begin{pmatrix}
			1\\
			x_\ti\\
			u_\ti
		\end{pmatrix}^\top R
		\begin{pmatrix}
			1\\
			x_\ti\\
			u_\ti
		\end{pmatrix} \label{eq:timeInvProblem1}\\
		\mathrm{s.t.} ~&~ x_{\ti +1} = f + A x_\ti + B u_\ti + w_\ti ,\\
		~&~ \bbE
		\lim_{N\to \infty} \frac{1}{N} \sum_{\ti =0}^N
		\begin{pmatrix}
			1\\
			x_\ti\\
			u_\ti
		\end{pmatrix}^\top H_i
		\begin{pmatrix}
			1\\
			x_\ti\\
			u_\ti
		\end{pmatrix} \leq 0, \quad i = 1,\ldots,s, \label{eq:timeInvProblem3}\\
    	~&~ u_\ti \sim \bbP_u,
        ~~~ w_\ti \sim \bbP_w,
		~~~ x_0 \sim \bbP_{x_0}.
	\end{align}
\end{subequations}
In this problem, we minimize the average cost \eqref{eq:timeInvProblem1}, since the summed cost might often be unbounded, and we impose the average constraints \eqref{eq:timeInvProblem3}. In addition, we assume that the system data $(f,A,B)$, our control policy $\bbP_u$, and the noise distribution $\bbP_w$ with variance $\Sigma^w$ are time-invariant. As in the time-varying case, we study the matrix function $\widetilde{F}$. This time, however, we seach for a time-invariant matrix
\begin{align}
	\Sigma =
	\begin{pmatrix}
		\sigma^{11} & \sigma^{12} & \sigma^{13}\\
		\sigma^{21} & \Sigma^{22} & \Sigma^{23}\\
		\sigma^{31} & \Sigma^{32} & \Sigma^{33}
	\end{pmatrix}
\end{align}
and define a time-invariant control policy according to
\begin{align}
	u_\ti &= k^1 + K^2 x_\ti + v_\ti \label{eq:tInvControllerReconstruction}
\end{align}
with controller parameters $K = \begin{pmatrix} k^1 & K^2 \end{pmatrix}$, $\Sigma^v$ satisfying
\begin{align}
 \begin{pmatrix}
		\sigma^{31} & \Sigma^{32}
	\end{pmatrix}
    &= 
    \begin{pmatrix} k^1 & K^2 \end{pmatrix}
	\begin{pmatrix}
		\sigma^{11} & \sigma^{12}\\
		\sigma^{21} & \Sigma^{22}
	\end{pmatrix}\\
    \Sigma^v &= \Sigma^{33} - \begin{pmatrix} k^1 & K^2 \end{pmatrix} \begin{pmatrix} \sigma^{13} & \Sigma^{23} \end{pmatrix}^\top \label{eq:tInvSigmaV}
\end{align}
and $v_\ti$ for $\ti \in \bbN_0$ being identically and independently distributed random variables, independent from $x_0$ and $w_\ti$ for $\ti \in \bbN_0$ with zero mean and variance $\Sigma^v$.
\begin{lemma}
	\label{lem:sigmaConvergence}
	Let $\Sigma \in \bbR^{(1+n+m)\times (1+n+m)}$ satisfy the conditions $\Sigma \succeq 0$, $\widetilde{F}(\Sigma,\Sigma,\Sigma^w) = 0$ and $\sigma^{11} = 1$. Then, if $\Sigma^w \succ 0$ holds and $u_\ti$ is chosen according to the policy \eqref{eq:tInvControllerReconstruction}, the sequence of moment matrices $(\Sigma_\ti)$ generated by the policy \eqref{eq:tInvControllerReconstruction} converges to $\Sigma$.
\end{lemma}
\begin{proof}
	Denote the equation $\widetilde{F}(\Sigma,\Sigma,\Sigma^w) = 0$ as
	\begin{align*}
		\begin{pmatrix}
			\sigma^{11} & \sigma^{12}\\
			\sigma^{21} & \Sigma^{22}
		\end{pmatrix}
		&\succeq
		\begin{pmatrix}
			1 & 0\\
			f^K & A^K
		\end{pmatrix}
		\begin{pmatrix}
			\sigma^{11} & \sigma^{12}\\
			\sigma^{21} & \Sigma^{22}
		\end{pmatrix}
		\begin{pmatrix}
			1 & 0\\
			f^K & A^K
		\end{pmatrix}^\top\\
		&+
		\begin{pmatrix}
			0 & 0\\
			0 & \Sigma^w + B\Sigma^v B^\top
		\end{pmatrix},
	\end{align*}
	where $f^K := f + Bk^1$ and $A^K = A + BK^2$. This Lyapunov inequality implies that $A^K$ is stable. Next consider $\widetilde{\Sigma}_\ti := \Sigma - \Sigma_\ti$. For this matrix $\tilde{\sigma}_\ti^{11} = 0$ holds for all $\ti$, since by assumption $\sigma^{11} = 1 = \sigma_\ti^{11}$. In addition, this matrix satisfies the Lyapunov equation
	\begin{align*}
		\begin{pmatrix}
			\tilde{\sigma}_{\ti +1}^{11} & \tilde{\sigma}_{\ti +1}^{12}\\
			\tilde{\sigma}_{\ti +1}^{21} & \widetilde{\Sigma}_{\ti +1}^{22}
		\end{pmatrix}
		=
		\begin{pmatrix}
			1 & 0\\
			f^K & A^K
		\end{pmatrix}
		\begin{pmatrix}
			\tilde{\sigma}_\ti^{11} & \tilde{\sigma}_\ti^{12}\\
			\tilde{\sigma}_\ti^{21} & \widetilde{\Sigma}_\ti^{22}
		\end{pmatrix}
		\begin{pmatrix}
			1 & 0\\
			f^K & A^K
		\end{pmatrix}^\top .
	\end{align*}
    The left lower block of this Lyapunov equation reads $\tilde{\sigma}_{\ti +1}^{21} = A^K\tilde{\sigma}_\ti^{21} + f^K\tilde{\sigma}_\ti^{11}$, which, by $\tilde{\sigma}_\ti^{11} = 0$ for all $\ti$ and stability of $A^K$ implies $\tilde{\sigma}_\ti^{21} \to 0$ for $\ti \to \infty$. Finally, inspecting the right lower block of this Lyapunov equation $\widetilde{\Sigma}_{\ti +1}^{22} = f^K\sigma_\ti^{11} (f^K)^\top + A^K\sigma_\ti^{21}(f^K)^\top + f^K\sigma_\ti^{12} (A^K)^\top + A^K \widetilde{\Sigma}_\ti^{22} (A^K)^\top$ and recalling $A^K$ being stable, $\tilde{\sigma}_\ti^{11}$ being zero for all $\ti$ and $\tilde{\sigma}_\ti^{21}$ converging to zero, we conclude that $\widetilde{\Sigma}_\ti^{22}$ converges to zero. In total, $\widetilde{\Sigma}_\ti$ converges to zero.
\end{proof}

Due to Lemma \ref{lem:sigmaConvergence}, the average cost and average constraints simplify under the control policy \eqref{eq:tInvControllerReconstruction} to
\begin{align*}
	\bbE
	\lim_{N\to \infty} \frac{1}{N} \sum_{\ti =0}^N
	\begin{pmatrix}
		1\\
		x_\ti\\
		u_\ti
	\end{pmatrix}^\top R
	\begin{pmatrix}
		1\\
		x_\ti\\
		u_\ti
	\end{pmatrix}
	&= \trace \Sigma R,\\
	\bbE\lim_{N\to \infty} \frac{1}{N} \sum_{\ti =0}^N
	\begin{pmatrix}
		1\\
		x_\ti\\
		u_\ti
	\end{pmatrix}^\top H_i
	\begin{pmatrix}
		1\\
		x_\ti\\
		u_\ti
	\end{pmatrix}
	&= \trace \Sigma H_i.
\end{align*}
Consequently, the time-invariant control problem \eqref{eq:timeInvProblem} can be solved as the convex program
\begin{align}
	\minimize_{(\Sigma)} ~&~ \trace \Sigma R \label{eq:timeInvSynthesis}\\
	\mathrm{s.t.} ~&~ \sigma^{11} = 1,~~\Sigma \succeq 0, \nonumber\\
	~&~ \widetilde{F}(\Sigma,\Sigma,\Sigma^w) = 0, \nonumber\\
	~&~ \trace \Sigma H_{i} \leq 0, & i = 1,\ldots,s.\nonumber
\end{align}

\begin{remark}
	Since, by Lemma \ref{lem:sigmaConvergence}, $\Sigma_\ti$ converges to the stationary solution $\Sigma$, the expectation constraint \eqref{eq:timeVaryingProblem3} is satisfied asymptotically and not just on average. Moreover, if $\Sigma_0$ happens to equal $\Sigma$, then this constraint is satisfied pointwise in time.
\end{remark}

\begin{remark}
	Also a mixture of the time-varying and time-invariant control problem can be studied. To this end, we may assume that the covariance matrix sequence becomes stationary, that is, $(\Sigma_\ti) = (\Sigma_0,\Sigma_1,\ldots,\Sigma_{N-1},\Sigma_N,\Sigma_N,\Sigma_N,\ldots)$. Then, if the problem parameters $(f_\ti,A_\ti,R_\ti,H_{\ti i})$ also become stationary for $\ti \geq N$, we could consider the constraints $\widetilde{F}(\Sigma_\ti,\Sigma_{\ti +1},\Sigma_\ti^w) = 0$ for  $\ti = 0,\ldots,N-1$, and $\widetilde{F}(\Sigma_N,\Sigma_N,\Sigma_N^w) = 0$ for the sequence $(\Sigma_\ti)$.
    
    Regarding a suitable cost, we mention that a stationary sequence $(\Sigma_\ti)$ permits, e.g., for a discounted cost
	\begin{align*}
		\bbE \sum_{\ti =0}^N \gamma^\ti\begin{pmatrix}
			1\\
			x_\ti\\
			u_\ti
		\end{pmatrix}^\top R
		\begin{pmatrix}
			1\\
			x_\ti\\
			u_\ti
		\end{pmatrix}
		=
		\sum_{\ti =0}^{N-1} \gamma^\ti \trace R \Sigma_\ti + \frac{\gamma^N \trace R\Sigma_N}{1-\gamma}
	\end{align*}
	with discount factor $\gamma \in[0,1[$ by choosing $R_\ti = \gamma^\ti R$ for $\ti = 0,\ldots,N-1$ and $R_N = \frac{\gamma^N}{1-\gamma} R$.
\end{remark}

\section{The moment matrix perspective and duality}
\label{sec:5}

Next, we explore how the convexification of state feedback synthesis problems given in Section \ref{sec:3} and \ref{sec:4} relates to standard strategies for state feedback synthesis. To this end, we study one of the most basic problems of state feedback synthesis, namely, quadratic stabilization. Particularly, an affine linear controller $K$ stabilizes (some equilibrium of) the closed loop $(f^K,A^K) := (f + Bk^1, A + BK^2)$ of an affine linear system with $K$ in the sense of Lyapunov if there exists a symmetric matrix $P$ such that the \emph{primal matrix inequalities}
\begin{align}
	(\bullet)^\top
	\left(\begin{array}{cc|cc}
		-p^{11} & -p^{12}\\
		-p^{21} & -P^{22}\\ \hline
		& & p^{11} & p^{12}\\
		& & p^{21} & P^{22}
	\end{array}\right)
	\begin{pmatrix}
		1 & 0\\
		0 & I\\
		\hline
		1 & 0\\
		f^K & A^K
	\end{pmatrix}
	\preceq 0 \label{eq:oldPrimal}
\end{align}
and $P \succ 0$ hold. Equivalently an affine linear controller $K$ stabilizes an affine linear system if there exists a symmetric matrix $\widetilde{P}$ such that the \emph{dual matrix inequalities}
\begin{align}
	(\bullet)^\top
	\left(\begin{array}{cc|cc}
		-\tilde{p}^{11} & -\tilde{p}^{12}\\
		-\tilde{p}^{21} & -\widetilde{P}^{22}\\ \hline
		& & \tilde{p}^{11} & \tilde{p}^{12}\\
		& & \tilde{p}^{21} & \widetilde{P}^{22}
	\end{array}\right)
	\begin{pmatrix}
		1 & (f^K)^\top\\
		0 & (A^K)^\top\\
        \hline
		1 & 0\\
		0 & I
	\end{pmatrix}
	\succeq 0 \label{eq:oldDual}
\end{align}
and $\widetilde{P} \succ 0$ hold. The matrices $P$ and $\widetilde{P}$ can be linked by
\begin{align*}
    P^{-1} = \begin{pmatrix}
        p^{11} & p^{12}\\
		p^{21} & P^{22}
    \end{pmatrix}^{-1}
    =
    \begin{pmatrix}
        \tilde{p}^{11} & \tilde{p}^{12}\\
		\tilde{p}^{21} & \widetilde{P}^{22}
    \end{pmatrix}
    =
    \widetilde{P}.
\end{align*}

The relations between \eqref{eq:oldPrimal} and \eqref{eq:oldDual}, and $P$ and $\widetilde{P}$ can be established by congruence transforms and the Schur complement, or by the following dualization lemma.

\begin{lemma}[derived from Lemma 10.2, \cite{scherer2000robust}]
	\label{lem:dualizationLemma}
	Let $M$ be a real, symmetric, non-singular matrix. Then the \emph{primal matrix inequalities}
	\begin{align*}
		\begin{pmatrix}
			I\\
			W
		\end{pmatrix}^\top
		M
		\begin{pmatrix}
			I\\
			W
		\end{pmatrix}
		&\preceq 0,
		&
		\begin{pmatrix}
			0\\
			I
		\end{pmatrix}^\top
		M
		\begin{pmatrix}
			0\\
			I
		\end{pmatrix}&\succ 0
	\end{align*}
	are equivalent to the \emph{dual matrix inequalities}
	\begin{align*}
		\begin{pmatrix}
			W^\top\\
			-I
		\end{pmatrix}^\top
		M^{-1}
		\begin{pmatrix}
			W^\top\\
			-I
		\end{pmatrix}
		&\succeq 0,
		&
		\begin{pmatrix}
			I\\
			0
		\end{pmatrix}^\top
		M^{-1}
		\begin{pmatrix}
			I\\
			0
		\end{pmatrix} &\prec 0.
	\end{align*}
\end{lemma}

The dual matrix inequality \eqref{eq:oldDual} can be convexified by the change of variables from $(P,K)$ to $(\widetilde{P},\widetilde{K})$, where $\widetilde{K}:=KP^{-1}$; we refer to \cite{boyd1994linear} for this convexification.

In $(\widetilde{P},\widetilde{K})$, the dual matrix inequality \eqref{eq:oldDual} reads
\begin{align}
    &\hspace{-2mm}
    \resizebox{\linewidth}{!}{$
	(\bullet)^\top
    \left(
	\begin{array}{ccc|cc}
		-\tilde{p}^{11} & -\tilde{p}^{12} & -(\tilde{k}^1)^\top\\
		-\tilde{p}^{21} & -\widetilde{P}^{22} & -(\widetilde{K}^2)^\top\\
		-\tilde{k}^1 & -\widetilde{K}^2 & -\widetilde{K}(\widetilde{P})^{-1}\widetilde{K}^\top\\
        \hline
		& & & \tilde{p}^{11} & \tilde{p}^{12}\\
		& & & \tilde{p}^{21} & \widetilde{P}^{22}
	\end{array}
    \right)
	\begin{pmatrix}
		1 & f^\top\\
		0 & A^\top\\
		0 & B^\top\\
        \hline
		1 & 0\\
		0 & I
	\end{pmatrix}$} \nonumber\\
    & \hspace{70mm}\succeq 0. \label{eq:reshapedDualMI}
\end{align}
Notice that the non-convexity in $\widetilde{K}(\widetilde{P})^{-1}\widetilde{K}^\top$ can easily be resolved making use of the Schur complement lemma. 

In addition, the representation \eqref{eq:reshapedDualMI} of the dual matrix inequality \eqref{eq:oldDual} reveals how established convexifying variable transforms relate to the one proposed in this article. Namely, the feasibility of \eqref{eq:reshapedDualMI} with $\widetilde{P} \succeq 0$ is equivalent to the feasibility of $\widetilde{F}(\Sigma,\Sigma,0) \succeq 0$ with $\Sigma \succeq 0$. One direction of this statement can be shown by choosing $\Sigma$ as
\begin{align}
    \begin{pmatrix}
		\sigma^{11} & \sigma^{12} & \sigma^{13}\\
		\sigma^{21} & \Sigma^{22} & \Sigma^{23}\\
		\sigma^{31} & \Sigma^{32} & \Sigma^{33}
	\end{pmatrix}
    &=
	\begin{pmatrix}
		-\tilde{p}^{11} & -\tilde{p}^{12} & -(\tilde{k}^1)^\top\\
		-\tilde{p}^{21} & -\widetilde{P}^{22} & -(\widetilde{K}^2)^\top\\
		-\tilde{k}^1 & -\widetilde{K}^2 & -\widetilde{K}(\widetilde{P})^{-1}\widetilde{K}^\top
	\end{pmatrix}. \label{eq:variableTransform}
\end{align}
For the reverse direction, we consider any $\Sigma \succeq 0$ with $\widetilde{F}(\Sigma,\Sigma,0) \succeq 0$. Next, perturb $\Sigma^{33}$ to project $\Sigma$ onto the image of the right hand side of \eqref{eq:variableTransform}. One can check that the latter does not lead to a violation of $\Sigma \succeq 0$ or $\widetilde{F}(\Sigma,\Sigma,0) \succeq 0$. This reverses the variable transform \eqref{eq:variableTransform}.

We emphasize that the variables $\widetilde{P}$ and $\widetilde{K}$ in which established variable transforms convexify the quadratic stabilization problem equal, by \eqref{eq:variableTransform}, blocks of the moment matrix $\Sigma$. This suggests that $\widetilde{P}$ and $\widetilde{K}$ should be interpreted as moment matrices. The fact that \eqref{eq:oldDual} is related to $\widetilde{F}(\Sigma,\Sigma,0) \succeq 0$ and not to $\widetilde{F}(\Sigma,\Sigma,0) = 0$ does not alter the fact that $\widetilde{P}$ and $\widetilde{K}$ can be interpreted as moment variables. Indeed, Theorem \ref{thm:controllerParametrization} can be proven with $\widetilde{F}(\Sigma,\Sigma,0) \succeq 0$ (or $\widetilde{F}(\Sigma,\Sigma,0) \preceq 0$) in \eqref{eq:cond} with the difference that, in this case, the sequence $\Sigma_\ti$ is an upper (lower) bound on the true moment matrices.

We mention that these relations to quadratic stabilization extend to state feedback $H_2$ synthesis. Indeed, it can be shown that solving \eqref{eq:timeInvProblem} for $R = \begin{pmatrix}
    0 & C & 0
\end{pmatrix}^\top \begin{pmatrix}
    0 & C & 0
\end{pmatrix}$ and $\Sigma^w = B^2 (B^2)^\top$ is equivalent to designing a state feedback controller minimizing the $H_2$-norm from an input with input matrix $B^2$ to an output with output matrix $C$; consider also \cite{kamgarpour2017infinite}.

\section{Example: Non-convex constraints}
\label{sec:6}

In this section we examine the possibility of considering non-convex constraints (avoiding non-convex regions in state space) with numerical examples. To this end, we study the stochastic optimal control problem
\begin{subequations}
    \label{eq:exampleProblem}
\begin{align}
    \minimize_{(\bbP_{u_\ti})} ~&~ \bbE \sum_{\ti =0}^{N-1} \left(\|x_\ti\|^2 + \|u_\ti\|^2\right) + 100\|x_N\|^2\\
    \mathrm{s.t.} ~&~ x_{\ti + 1} = x_\ti + u_\ti , \label{eq:exDynamics}\\
    ~&~ \bbE \|u_\ti\|^2 \leq 0.1, \label{eq:exInputConstraint}\\
    ~&~ \bbE \|x_\ti - \hat{x}_i\|^2 \geq (r_i + \varepsilon)^2, ~~~ i = 1,\ldots,s, \label{eq:exObstacles}\\
    ~&~ u_\ti \sim \bbP_{u_\ti}, \quad x_0 \sim \delta_{\bar{x}} .
\end{align}
\end{subequations}
We assume that the states $x_t$ and control inputs $u_t$ take values in $\bbR^2$. The system dynamics \eqref{eq:exDynamics} define a two-dimensional integrator. We might imagine that this integrator is a very simple model for a robot in a plane. In this case, the constraint \eqref{eq:exInputConstraint} restricts the speed of the robot. Moreover, the constraints \eqref{eq:exObstacles} model that the robot should maintain a distance of $r_i$ from the points $\hat{x}_i$ for $i = 1,\ldots,s$.
The initial distribution is the dirac distribution $\bbP_{x_0} = \delta_{\bar{x}}$ with $\bar{x} = \begin{pmatrix}
	10 & 0
\end{pmatrix}^\top$ and the moments
\begin{align*}
	\begin{pmatrix}
		\bar{\sigma}_0^{11} & \bar{\sigma}_0^{12}\\
		\bar{\sigma}_0^{21} & \ovl{\Sigma}_0^{22}
	\end{pmatrix}
	=
	\begin{pmatrix}
		1 & \bar{x}^\top\\
		\bar{x} & \bar{x}\bar{x}^\top
	\end{pmatrix}.
\end{align*}
The cost reflects that the robot should reach $\begin{pmatrix}
	0 & 0
\end{pmatrix}^\top$.
The problem \eqref{eq:exampleProblem} is an instance of \eqref{eq:timeVaryingProblem} for appropriate choices of $R_\ti$ and $H_{\ti i}$ for $\ti = 0,\ldots,N$ and $i = 1,\ldots,s$.

\subsection{Test 1}

For a first test, we consider two constraints \eqref{eq:exObstacles} with
\begin{align*}
	\begin{pmatrix}
		\hat{x}_{11}\\
		\hat{x}_{12}
	\end{pmatrix}
	&=
	\begin{pmatrix}
		-7.5\\
		0.5
	\end{pmatrix},
	&
	\begin{pmatrix}
		\hat{x}_{21}\\
		\hat{x}_{22}
	\end{pmatrix}
	&=
	\begin{pmatrix}
		-2.5\\
		-0.5
	\end{pmatrix}
\end{align*}
and $r_1 = r_2 = 1$. For these constraints, we map \eqref{eq:exampleProblem} to the semi-definite program \eqref{eq:timeVaryingSynthesis} and solve the latter to obtain a control policy $(\bbP_{u_\ti})$. We then simulate the closed loop of this policy with the dynamics \eqref{eq:exDynamics}. Ten trajectories of this closed loop are depicted in Figure \ref{fig:twoObstacles}. As we can see, the trajectories are almost not distinguishable. Indeed, the solution of this problem is an (almost) deterministic policy. The constraints are then also satisfied deterministically. Note that we have chosen $\varepsilon = 0.1$ in \eqref{eq:exObstacles} to create a visible margin between the areas to be avoided and the trajectory. The code for this example can be accessed via \url{https://github.com/SphinxDG/MomentRelaxationsControllerSynthesisAndNonConvexConstraints/tree/main}.

Assuming that we wish to satisfy the constraints \eqref{eq:exObstacles} deterministically, we can consider this first example as a case where the relaxation works well. Unfortunately, there are also cases, where the controller is not deterministic and constraints are not deterministically satisfied.

\begin{figure}
	\centering
	\input{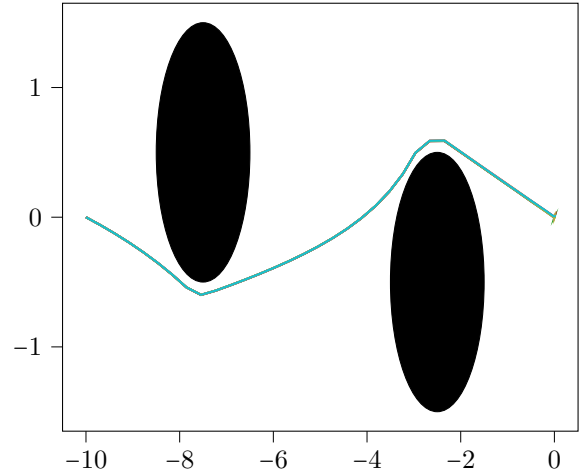}
	\caption{Vizualization of trajectories generated by the random policy \eqref{eq:controllerReconstruction} and two black areas to be avoided.}
	\label{fig:twoObstacles}
    \vspace{-3mm}
\end{figure}

\subsection{Test 2}

In a second test, we consider one constraint \eqref{eq:exObstacles} with $\hat{x}_1 = \begin{pmatrix}
	\hat{x}_{11} & \hat{x}_{12}
\end{pmatrix} = \begin{pmatrix}
-5 & 0
\end{pmatrix}$ and $r_1 = 1$, i.e., the area to be avoided is exactly in between the starting position of our system and the desired target.

If we solve the control problem \eqref{eq:exampleProblem} for this constraint, we obtain closed-loop trajectories resembling the ten trajectories depicted in Figure \ref{fig:oneObstacle}. This scenario presents a significantly different case compared to the previous example. The variance in the trajectories highlights that the policy is stochastic, and seven out of ten trajectories pass through the area to be avoided. If \eqref{eq:exampleProblem} results from the relaxation of a problem with hard constraints, then this example yields an inexact relaxation. This behavior is feasible for the control problem \eqref{eq:exampleProblem}, since we are optimizing over stochastic policies and impose only expectation constraints. Hence, we see that the solution of \eqref{eq:timeVaryingProblem} can yield a stochastic policy if \eqref{eq:timeVaryingProblem} includes non-convex costs or constraints. Moreover, in the stochastic case, some or even all of the trajectories sampled from the closed loop can violate the constraints. We mention that the situation of an \emph{obstacle} exactly in between a robot and a target is sometimes considered to be especially challenging since avoiding the obstacle by going to the left or going to the right seems equally preferable; consider, e.g., the elaboration on indecision in \cite{cathcart2023proactive}, where avoidance strategies are examined just for this scenario. Incidentally, the solution to the current example becomes deterministic if perturbing $\hat{x}_1$ by a tiny amount in the $x_2$ direction.

\begin{figure}
	\centering
	\input{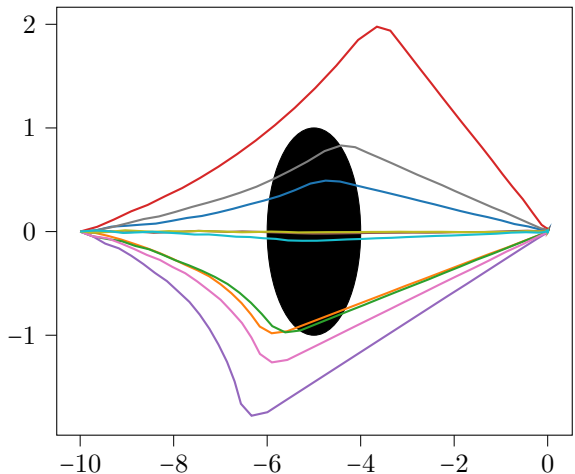}
	\caption{Vizualization of trajectories generated by the random policy \eqref{eq:controllerReconstruction} for an obstacle avoidance problem with one obstacle.}
	\label{fig:oneObstacle}
    \vspace{-3mm}
\end{figure}

Of course, this example shows that caution is required when using such a relaxation in practice. For non-convex problems, we cannot rely on obtaining a deterministic solution. For problems with soft constraints, the risk of a stochastic solution may be acceptable if the latter occurs rarely. For problems with hard constraints, as in obstacle avoidance, the risk may not be acceptable. Nevertheless, we believe that the methodology presented here may be useful for the iterative solution of such problems, e.g., for the initialization of solvers.

\section{Example: Fixed-point escape}
\label{sec:7}

A first example of the flexibility of the moment matrix approach to \eqref{eq:timeInvProblem} that resembles obstacle avoidance is given in Section \ref{sec:6}. As a second example, we now consider the swing up of an inverted pendulum. The dynamics of an inverted pendulum can be described by the system of nonlinear differential equations \cite{khalil2002control}
\begin{align}
    \begin{pmatrix}
        \dot{x}_1\\
        \dot{x}_2\\
        \dot{x}_3\\
        \dot{x}_4
    \end{pmatrix}
    =
    \begin{pmatrix}
        x_3\\
        x_4\\
        \frac{m_2lx_4^2\sin x_2 -m_2 g\sin x_2 \cos x_2 + u}{m_1 + (\sin x_2)^2 m}\\
        \frac{(m_2lx_4^2 \cos x_2 - (m_1+m_2)g)\sin x_2 - (l\cos x_2)u}{l(m_1 + (\sin x_2)^2 m_2)}
    \end{pmatrix}. \label{eq:invertedPendulum}
\end{align}
In this system of equations, $g = 9.81$ is the gravitational acceleration and $m_1 = 1$ is the mass of a cart on which a pendulum with mass $m_2 = 10^{-3}$ rests. The input $u$ is a force acting on the cart. The state $x_1$ is the position of the cart, $x_2$ is the angle of the pendulum, and $x_3$ and $x_4$ are the respective velocity and angular velocity.

To solve the swing up task, we design a controller that swings the pendulum high enough to reach the region of attraction of a second linear controller that stabilizes the upper pendulum position. The controller that makes the pendulum \emph{escape} from the lower equilibrium is obtained as the solution to the control problem
\begin{subequations}
\begin{align}
    \minimize_{(\bbP_{u})} ~&~ \lim_{N\to\infty} \frac{1}{N} \bbE \sum_{\ti = 0}^{N-1} (x_{1\ti}^2 + x_{3\ti}^2 - 10^4 e(x_{2\ti},x_{4\ti}) + u_{\ti}^2) \nonumber\\
    \mathrm{s.t.}~&~ x_{1\ti+1} = x_{1\ti} + h x_{3\ti}, \nonumber\\
    ~&~ x_{2\ti+1} = x_{2\ti} + h x_{4\ti}, \nonumber\\
    ~&~ x_{3\ti+1} = x_{3\ti} - h \frac{m_2g}{m_1}x_{2\ti} + h \frac{1}{m_2}u, \nonumber\\
    ~&~ x_{4\ti+1} = x_{4\ti} - h \frac{(m_1+m_2)g}{lm_2}x_{2\ti} + h \frac{l}{m_1}u, \nonumber\\
    ~&~ \bbE e(x_{2\ti},x_{4\ti}) \leq e(2,0), \label{eq:energyConstraint}\\
    ~&~ \bbE u_{\ti}^2 - \bbE u_{\ti}x_{\ti}^\top (\bbE x_\ti x_\ti^\top)^{-1} \bbE x_{\ti}u_{\ti} \geq 10^4 h, \label{eq:excitationConstraint}
\end{align}
\end{subequations}
which is (almost) an instance of the infinite-horizon design problem \eqref{eq:timeInvProblem}. The dynamic constraint of this control problem is obtained from a linearization of \eqref{eq:invertedPendulum} in the lower equilibrium and a subsequent Euler discretization. Unlike \eqref{eq:timeInvProblem}, this problem incorporates the additional constraint \eqref{eq:excitationConstraint}, which ensures that the designed control policy excites the pendulum sufficiently. Without this constraint, the optimal solution of the above problem would yield the zero policy. Fortunately, this constraint can be convexified in $\Sigma$ using the Schur complement. We also highlight the term $e(x_{2\ti},x_{4\ti}) = \frac{1}{2} mgl x_{2\ti}^2 + \frac{1}{2} gl^2 x_{4\ti}^2$ which enters negatively into the cost. This term is a quadratic approximation of the pendulum energy $mgl (1-\cos(x_{2\ti})) + \frac{1}{2} gl^2 x_{4\ti}^2$ making sure that the control policy pursues the goal of increasing the energy of the pendulum. Note that once the pendulum is supplied with enough energy, it will eventually reach the upper position. The energy constraint \eqref{eq:energyConstraint} caps the average pendulum energy to the value of the potential energy of a pendulum at an angle of $2 \mathrm{rad} \approx 115^\circ$ preventing the pendulum energy from going to infinity. 

Two exemplary closed loop trajectories of the proposed control policy with the nonlinear pendulum model \eqref{eq:invertedPendulum} are depicted in Figure \ref{fig:angle} and Figure \ref{fig:sledPos}. First, the policy for the escape from the lower equilibrium is active and supplies the pendulum with energy. Then, the pendulum reaches the region of attraction of a second controller and is stabilized in the upward position. The moment for the switching of control policies is determined based on a Laypunov function for the second controller.

The presented control strategy for the swing up of an inverted pendulum stands out by being composed of two linear controllers. This means that we can take advantage of the extensive possibilities for tuning linear controllers. For example, we can add additional constraints to the optimal control problem, such as restrictions on the position and speed of the cart. We also have the possibility of including a frequency-dependent cost in the optimization problem \cite{stein1987lqg}. This allows us to put a penalty on high frequency control inputs. Finally, we mention that the controller in the form of two linear controllers requires only minimal online computation.

\begin{figure}
	\centering
	\input{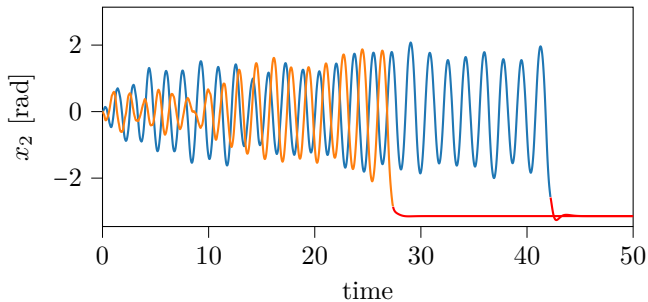}
	\vspace{-4mm}
	\caption{Two closed loop trajectories of the angle of the nonlinear inverted pendulum model \eqref{eq:invertedPendulum} with a swing up controller. Where the color of the trajectories is orange or blue, an escape controller for the lower equilibrium is used. Where the color of the trajectories is red, a stabilizing controller for the upper equilibrium is used.}
	\label{fig:angle}
\end{figure}

\begin{figure}
	\centering
	\input{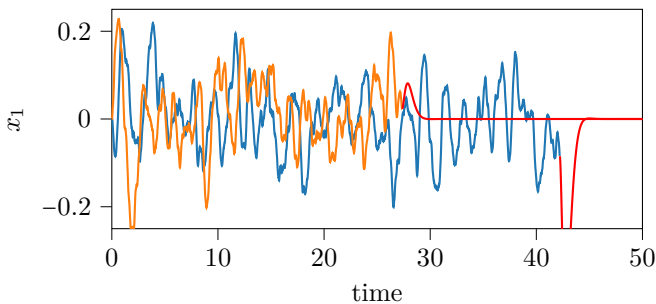}
	\vspace{-3mm}
	\caption{Two closed loop trajectories of the cart position of the nonlinear inverted pendulum model \eqref{eq:invertedPendulum} with a swing up controller. Where the color of the trajectories is orange or blue, an escape controller for the lower equilibrium is used. Where the color of the trajectories is red, a stabilizing controller for the upper equilibrium is used.}
	\label{fig:sledPos}
    \vspace{-4mm}
\end{figure}

\section{Conclusion}

We present a convexification strategy for linear state feedback synthesis problems with affine, time-varying system dynamics, random initial state, and additive stochastic noise. This convexification is based on moment matrices and permits for non-convex quadratic costs and constraints. Like second-order moment relaxations in non-convex quadratic programming, the case we study always permits for the extraction of a \emph{solution}, which is actually optimal if we relax hard constraints to expectation constraints. In addition, we identify the parameters of known convexification strategies based on the dualization lemma with blocks of our moment matrices. Hence, we interpret the new decision variables after applying the dualization lemma as moment variables.

\bibliographystyle{abbrv}
\bibliography{sources.bib}
\end{document}